\documentclass{article}
\usepackage{latexsym,amsfonts,palatino,eulervm,color,amsmath,amssymb,amsthm,graphics,graphicx,comment,float}
\usepackage[hang,flushmargin]{footmisc}
\newcommand\blfootnote[1]{%
  \begingroup
  \renewcommand\thefootnote{}\footnote{#1}%
  \addtocounter{footnote}{-1}%
  \endgroup
}

\newtheorem{theorem}{Theorem}[section]
\newtheorem{lemma}[theorem]{Lemma}

\newtheorem{example}[theorem]{Example}

\def\n{\noindent}
\def\GL{\mathrm{GL}}
\def\Q{\mathbb{Q}}

\def\C{\mathbb{C}}
\def\R{\mathbb{R}}

\def\lcm{\mathrm{lcm}}

\def\A{\mathcal{A}}
\def\S{\mathcal{S}}
\def\Bcal{\mathcal{B}}

\title{Alternating sign matrices of finite multiplicative order}
\author{Cian O'Brien \and Rachel Quinlan}

\begin{document}
\maketitle
\begin{abstract}
    We investigate alternating sign matrices that are not permutation matrices, but have finite order in a general linear group. We classify all such examples of the form $P+T$, where $P$ is a permutation matrix and $T$ has four non-zero entries, forming a square with entries $1$ and $-1$ in each row and column. We show that the multiplicative orders of these matrices do not always coincide with those of permutation matrices of the same size. We pose the problem of identifying finite subgroups of general linear groups that are generated by alternating sign matrices.
\end{abstract}

\n\emph{Keywords}: Alternating sign matrix, minimum polynomial.\\
\n\emph{MSC}: 15B36, 05C50, 05C20

\blfootnote{Cian O'Brien, School of Mathematics, Cardiff University, \texttt{obrien.cian@outlook.com} \\
Rachel Quinlan, School of Mathematical and Statistical Sciences, National University of Ireland, Galway, \texttt{rachel.quinlan@nuigalway.ie}}

\section{Introduction}

An \emph{alternating sign matrix (ASM)} is a $(0,1,-1)$-matrix with the property that the non-zero entries in each row and column alternate in sign, beginning and ending with $+1$. Alternating sign matrices were first investigated by Mills, Robbins, and Rumsey \cite{MRR}, in a context arising from the classical theory of determinants. Connections to fields such as statistical mechanics \cite{Zeilberger} and enumerative combinatorics \cite{Kuperberg} were subsequently discovered, and ASMs continue to attract sustained interest from diverse viewpoints. We refer to Bressoud's book \cite{Bressoud}, for a comprehensive account of the emergence of attention to ASMs and the mathematical developments that ensued. 

A recurrent theme in the study of ASMs is their occurrence, in independent contexts, as natural generalizations of permutation matrices. This invites the question of whether and how familiar themes in the study of permutations can be applied or adapted to ASMs. For example, ASMs first emerged in the definition of the $\lambda$-determinant of a square matrix, which involves adapting the technique of Dodgson condensation by replacing the usual $2\times 2$ determinant with a version involving a parameter $\lambda$. Alternating sign matrices play the role for the $\lambda$-determinant that permutations do for the special case of the classical determinant, which arises of the value of $\lambda$ is set to 1. Lascoux and Sch\"{u}tzenberger showed in \cite{bruhat} that the set of $n\times n$ ASMs is the unique minimal lattice extension of the set of $n\times n$ permutation matrices under the \emph{Bruhat partial order}. An extension of the concept of Latin squares, which arise by replacing permutation matrices with ASMs, is investigated in \cite{ashl1} and \cite{ashl2}.

Our focus in this article is on non-singular ASMs with the special property of having finite order as elements of the general linear group. This topic connects to the position of permutations among all ASMs, and also to some recent attention in the literature to the behaviour on ASMs of algebraic invariants such as the spectral radius, characteristic polynomial and Smith normal form \cite{spectra1,spectra2, snf}.

In \cite{spectra2}, Brualdi and Cooper study the maximum possible spectral radius of an ASM. They note that the minimum spectral radius of an ASM is more easily identified, since every ASM has common row sum 1 and hence has 1 as an eigenvalue; moreover the permutation matrices are examples of ASMs whose eigenvalues all have modulus 1. The following example is presented in \cite{spectra2}, to show that the minimum possible spectral radius of 1 may also occur in the case of an ASM that includes negative entries.

\begin{example}\label{finite} The matrix
$$
A = \left(\begin{array}{rrrrr}
0 & 0 & 1 & 0 & 0 \\
1 & 0 & -1 & 1 & 0 \\
0 & 0 & 1 & -1 & 1 \\
0 & 0 & 0 & 1 & 0 \\
0 & 1 & 0 & 0 & 0
\end{array}\right)
$$
is an ASM satisfying $A^6=I_5$, the $5\times 5$ identity matrix. Its characteristic polynomial is $(x-1)^2(x+1)(x^2-x+1)$ and its minimum polynomial is $(x-1)(x+1)(x^2-x+1)$. 
\end{example}

Within the set $\A_n$ of all $n\times n$ ASMs, the set $\S_n$ of permutation matrices is a multiplicative group of $n!$ elements. In the following lemma, we observe that a set of ASMs that is a group under matrix multiplication must consist of permutation matrices.

\begin{lemma}\label{group}
Suppose that $A$ and $B$ are $n\times n$ ASMs that satisfy $AB=I_n$. Then $A$ and $B$ are permutation matrices. 

\begin{proof}
The first row of $A$ has a 1 as its only nonzero entry; suppose that this occurs in position $j$. Then a 1 in the $(j,1)$ position is the only nonzero entry of Column 1 of $B$. Every subsequent column of $B$ is orthogonal to Row 1 of $A$, so Row $j$ of $B$ has only zeros after its first entry. Similarly, Rows $2,\dots ,n$ of $A$ are all orthogonal to Column 1 of $B$, so the only nonzero entry of Column $j$ of $A$ is the first. Deleting Row 1 and Column $j$ from $A$, and deleting Column 1 and Row $j$ from $B$, leaves a pair of matrices $A'$ and $B'$ in $\A_{n-1}$ that satisfy $A'B'=I_{n-1}$. The conclusion follows by induction on $n$.
\end{proof}
\end{lemma}

While Lemma \ref{group} eliminates the possibility that $\A_n$ could contain multiplicative groups other than subgroups of $\S_n$, Example \ref{finite} demonstrates the existence of finite multiplicative groups that are \emph{generated} by (non-permutation) ASMs. We remark that the group generated by the matrix $A$ of Example \ref{finite} is isomorphic to a subgroup of the symmetric group $S_5$. However, since every element of order 6 in $S_5$ consists of a $2$-cycle and a $3$-cycle, disjoint from each other, every $5\times 5$ permutation matrix of order $6$ has characteristic polynomial $(x^2-1)(x^3-1) = (x-1)^2(x+1)(x^2+x+1)$. Thus the matrix $A$ of Example \ref{finite} is not similar to a permutation matrix.

One may pose the question of which finite subgroups of $\GL(n,\R)$ are generated by invertible alternating sign matrices, and which such groups do not have isomorphic copies within $S_n$. In this article, we consider the case of finite cyclic subgroups and investigate elements of finite multiplicative order in a particular subset of $\A_n$.

In order to describe the class of ASMs of interest, we introduce the notion of a $T$-block, adapted from \cite{asbg1, asbg2}. A T-block is a $n\times n$ matrix whose non-zero entries form a (not necessarily contiguous) copy of
\[
\pm\left(\begin{array}{rr}
1 & -1 \\
-1 & 1
\end{array}\right).\]



We denote by $T(i_1,j_1,i_2,j_2)$ the $T$-block with $1$ in positions $(i_1,j_1)$ and $(i_2,j_2)$, and $-1$ in positions $(i_1, j_2)$ and $(i_2, j_1)$, where $i_1\neq i_2$ and $j_1\neq j_2$. We remark that this notational designation implies that $T(i_1,j_1,i_2,j_2)=T(i_2,j_2,i_1,j_1)$. Whenever the situation is sufficiently specified, we will choose the version with $i_1<i_2$. The following assertion is essentially Theorem 6.2 of \cite{asbg1}.

\begin{theorem}\label{ASM interchange}
Every $n \times n$ ASM can be obtained from the identity matrix $I_n$ through a sequence of additions of $T$-blocks, in such a way that an ASM is obtained at every step.
\end{theorem}
An extension of Theorem \ref{ASM interchange} to $n\times n \times n$ alternating sign hypermatrices appears in \cite{ashl2}.

In this article, we consider ASMs that differ from a permutation matrix by the addition of a single $T$-block, having the form $P+T$ for a permutation matrix $P$ and $T$-block $T$. We refer to any matrix of this form as a $PT$-matrix. While the class of $PT$-matrices includes all permutation matrices, our attention will be focussed on non-permutation $PT$-matrices. The goal of this article is to identify all ASMs of multiplicative order that are $PT$-matrices, up to permutation similarity and transposition (where the matrices $A$ and $B$ are \emph{permutation similar} if $B=P^TAP$ for a permutation matrix $P$). We note that the properties of being a permutation matrix, a $T$-block, or a $PT$-matrix, are all preserved under conjugation by a permutation matrix. This is not generally true of an ASM however.

In Section \ref{algebra}, we recall some properties of rational matrices of finite multiplicative order. In Section \ref{graphsection}, we introduce the directed graph of a $PT$-matrix and use graph-theoretic considerations to identify candidates for finite multiplicative order. In Section \ref{polynomials}, analysis of the minimum polynomials of matrices determined by these candidate graphs leads to a complete description of $PT$-matrices of finite order. In Section 5, we show that all but a few exceptions are permutation similar to alternating sign matrices.

\section{Rational matrices of finite multiplicative order}\label{algebra}
In this section we recall some relevant properties of the characteristic and minimum polynomials of matrices of finite order in $\GL (n,\Q )$. For information on the minimum polynomial of a matrix, the companion matrix of a polynomial, and related algebraic background, we refer to Chapter 3 of \cite{Jacobson}.

For a positive integer $d$, we write $\Phi_d(x)$ for the $d$th cyclotomic polynomial, the monic polynomial in $\mathbb{Z}[x]$ whose roots are the primitive roots of unity of order $d$ in $\C$. Then $\Phi_d(x)$ is irreducible in $\Q [x]$ and its degree is $\phi (d)$, where $\phi$ denotes the Euler totient function. 

Suppose that $A\in\GL(n,\Q )$ has multiplicative order $t$. Then $A^t-I_n=0$, and so the minimum polynomial $m_A(x)$ of $A$ divides $x^t-1$ in $\Q[x]$. It follows that $m_A(x)$ is a product of distinct cyclotomic polynomials $\Phi_d(x)$, where $d$ runs through a set of divisors of $t$ whose least common multiple is $t$. On the other hand, any matrix whose minimum polynomial has this form does have finite order, equal to the least common multiple of the orders of its roots in $\C^\times$. 

The possible finite orders of elements of $\GL (n,\Q)$ are integers of the form $\lcm (d_1,\dots,d_k)$, where the $d_i$ are positive integers with 
$$
\sum_{i=1}^k \phi(d_i)=n.
$$
For example, the possible finite orders of elements of $\GL (5,\Q)$ are $1, 2, 3, 4, 5, 6, 8, 10, 12$. The possible orders of $n\times n$ permutation matrices are those integers that occur as the least common multiple of the parts in a partition of $n$. In the case $n=5$, these are $1,2,3,4,5$ and $6$.

For the $PT$-matrices of interest in this article, the characteristic polynomial is generally more easily computed than the minimum polynomial. Both have the same irreducible factors, but they may occur with higher multiplicity in the characteristic polynomial. We will identify $PT$-matrices with the property that their characteristic polynomial is a product of cyclotomic factors. For every $d\ge 2$, the polynomial $\Phi_d(x)$ is \emph{palindromic}, meaning that the sequence of its coefficients remains unchanged when reversed, or equivalently that $\Phi_d(x)=x^{\phi(d)}\Phi_d(\frac{1}{x})$. The polynomial $\Phi_1(x)=x-1$ is \emph{skew-palindromic}; reversing the sequence of its coefficients negates each term. Every product of palindromic and skew-palindromic polynomials is itself palindromic or skew-palindromic, according as the number of skew-palindromic factors is even or odd. Thus we may restrict our attention to  $PT$-matrices with palindromic or skew-palindromic characteristic polynomials.

We recall that a square matrix in $\GL(n,\C)$ is diagonalizable in $\GL (n,\C)$ if and only if its minimum polynomial has distinct roots. It follows that every rational square matrix of finite order is diagonalizable over $\C$, since its minimum polynomial divides $x^t-1$ for some $t$. Indeed a matrix whose characteristic polynomial is a product of cyclotomic polynomials has finite multiplicative order if and only if it is diagonalizable. This observation will be useful at times in Sections \ref{graphsection} and \ref{polynomials}.  In a case where the characteristic polynomial has no repeated irreducible factor, the characteristic and minimum polynomials coincide and the matrix is diagonalizable.

Given a monic polynomial $p(x)=x^n+a_{n-1}x^{n-1}+\dots +a_1x+a_0$, we define the \emph{companion matrix} of $p(x)$ to be the $n\times n$ matrix $C$ that has $1$ in the $(i+1,i)$-position for $1\le i\le n-1$, has the entries $-a_0,-a_1,\dots ,-a_{n-1}$ in Column $n$ and has zeros in all other positions. Then $p(C)=0_{n\times n}$ and $p(x)$ is the minimum polynomial (and the characteristic polynomial) of $C$. For a positive integer $k$, we write $C_k$ for the companion matrix of the polynomial $x^k-1$. We note that $C_k$ is a permutation matrix, representing a cycle of length $k$.

\section{Graphs and $(0,1,-1)$-matrices}\label{graphsection}
We associate a 2-arc-coloured directed graph $\Gamma_A$ to a $n\times n$ $(0,1,-1)$-matrix $A$ as follows. The vertex set of $\Gamma_A$ is $\{v_1,\dots ,v_n\}$ and the coloured arcs are as follows:
\begin{itemize}
    \item $(v_i,v_j)$ is a blue arc if $A_{ij}=1$;
    \item $(v_i,v_j)$ is a red arc if $A_{ij}=-1$;
    \item $(v_i,v_j)$ is not an arc if $A_{ij}=0$.
\end{itemize}
The same interpretation of arcs and entries yields an association of a square $(0,1,-1)$-matrix to a given 2-arc-coloured digraph, upon the choice of an ordering of the vertices. Each graph corresponds to a permutation equivalence class of matrices. 

We refer to a 2-arc-coloured digraph that is associated to a $PT$-matrix by the above correspondence, as a \emph{$PT$-graph}. Every $PT$-graph has at most two red arcs, since a $PT$-matrix has at most two negative entries. A $PT$-matrix with no negative entry is a permutation matrix, whose graph consists of disjoint directed cycles, with all arcs coloured blue. A $PT$-matrix with exactly one negative entry has the form $P+T$, where exactly one of the negative entries of the $T$-block $T$ occurs in the same position as a 1 in the permutation matrix $P$. If this occurs in Row $i$, then Row $i$ of $P+T$ is a duplicate of another row, and hence $P+T$ has zero determinant and cannot have finite multiplicative order. Since our interest is in $PT$-matrices that have finite order and are not permutations, we may confine our attention to $PT$-graphs that have exactly two red arcs. While is it possible for a $PT$-graph to have a double blue arc, such a graph corresponds to a $PT$-matrix with an entry equal to 2, which cannot be permutation similar to an ASM. Our concern is thus with $PT$-graphs having exactly two red arcs and no multiple arcs, corresponding to matrices of the form $P+T$, where the positions of non-zero entries in the permutation matrix $P$ and the $T$-block $T$ do not coincide. If $\Gamma$ is such a graph, with $n$ vertices, then its arcset is the disjoint union of a set of four arcs (two of each colour) corresponding to the entries of a $T$-block, and a set of $n$ blue arcs corresponding to the entries of a permutation matrix. These two sets are uniquely determined by the two red arcs. We write $\Gamma_T$ for the subgraph of $\Gamma$ consisting of the four arcs determined by entries of $T$ and their incident vertices, and $\Gamma_P$ for the subgraph similarly determined by the arcs arising from $P$. The vertex set of $\Gamma_P$ is the same as that of $\Gamma$, and its $n$ arcs comprise disjoint directed cycles. The arc set of $\Gamma_T$ is disjoint from that of $\Gamma_P$, and $\Gamma_T$ has one of the following forms.
\begin{center}
    \includegraphics[width=\textwidth]{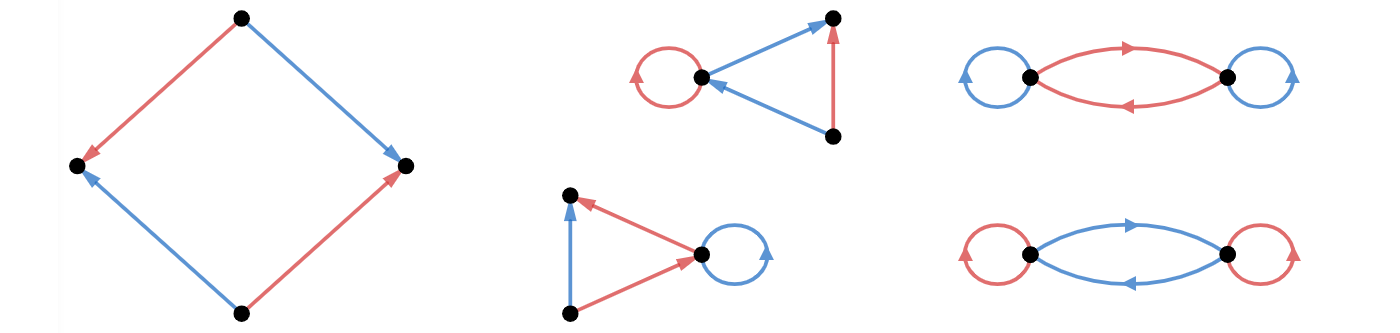}
\end{center}
The graph $\Gamma_T$ is \emph{weakly connected}, meaning that its underlying undirected graph is connected. The weakly connected component of $\Gamma$ that includes $\Gamma_T$ involves arcs from at most four cycles of $\Gamma_P$, since each vertex of $\Gamma_T$ occurs in one cycle of $\Gamma_P$. Since any additional weakly connected components of $\Gamma$ are directed cycles, a $(0,1,-1)$-matrix corresponding to $\Gamma$ has finite multiplicative order if and only if the submatrix corresponding to the weakly connected component that includes $\Gamma_T$ does. For this reason, for the remainder of this section we only consider weakly connected $PT$-graphs. We consider separately the cases where the vertices of $\Gamma_T$ are incident with one, two, three or four cycles of $\Gamma_P$.

The \emph{reverse} of a digraph $\Gamma$ is the graph obtained from $\Gamma$ by reversing the directions of all arcs, while maintaining any arc colouring . The operation of reversing the arc directions in a two-arc-coloured digraph has the effect of transposing the corresponding $(0,1-1)$-matrix. Since the matrix property of having finite multiplicative order is preserved under transposition, this observation is useful in limiting the number of graph types requiring analysis. 

Given any digraph $\Gamma$ with arcs coloured red and blue, we say that a walk in $\Gamma$ is \emph{negative} if it includes an odd number of red arcs (counted with repetition), and \emph{positive} if it includes an even number of red arcs. For a positive integer $k$, we write $w_k^+(u,v)$ and $w_k^-(u,v)$ respectively for the numbers of positive and negative walks of length $k$ (abbreviated to $k$-walks) from the vertex $u$ to the vertex $v$ in $\Gamma$. Let $A$ be the $(0,1,-1)$-matrix determined by the ordering $v_1,\dots ,v_n$ of the vertices of $\Gamma$. It is routine to show that for a positive integer $k$, the entry in the $(i,j)$-position of $A^k$ is
\begin{equation}\label{walks}
w_k^+(v_i,v_j)-w_k^-(v_i,v_j).
\end{equation}
Suppose that $A^k=I_n$, for a positive integer $k$. Then the numbers of positive and negative $k$-walks from $u$ to $v$ in $\Gamma$ coincide, for any pair $u$ and $v$ of distinct vertices. For any vertex $u$, the number of positive $k$-walks from $u$ to $u$ exceeds the number of negative $k$-walks by 1. By applying these observations to directed graphs corresponding to $PT$-matrices, we will be able to reduce to four general classes of weakly connected $PT$-graphs, whose corresponding $PT$-matrices include all examples of finite multiplicative order that are not permutations, up to permutation equivalence and transposition.

\subsection{Type 1: a single cycle} 
We refer to $PT$-graphs and matrices involving a permutation with a single cycle as being of \emph{type 1}.

\includegraphics[width=\textwidth]{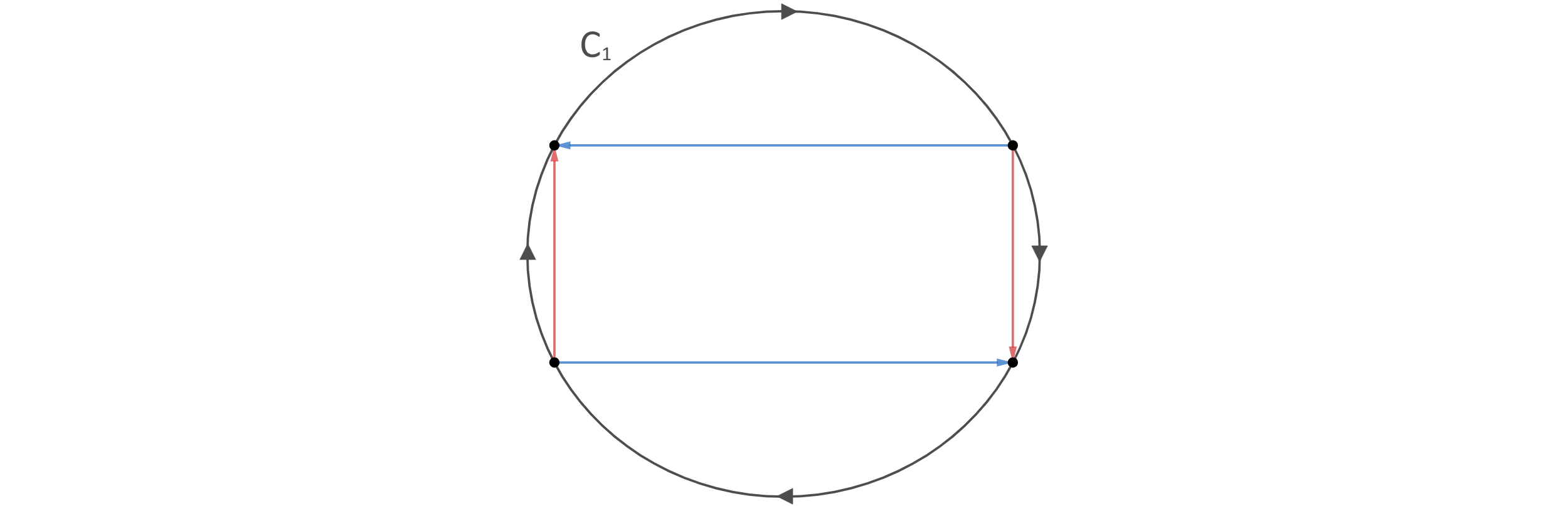}
Necessary and sufficient conditions for a $PT$-matrix of type 1 to have finite multiplicative order are established in Section \ref{type 1}. \\

\subsection{Type 2: a pair of cycles}
The distinct $PT$-graph structures (up to arc reversal, or equivalently, matrix transposition), where the vertices of $\Gamma_T$ belong to exactly two cycles of $\Gamma_P$, are depicted below. We refer to these graphs and their corresponding matrices as \emph{type 2(a)}, \emph{type 2(b)}, \emph{type 2(c)}, and \emph{type 2(d)}, respectively. In the case of type 2(d), the relative positions of the three vertices of $\Gamma_T$ along the directed cycle $C_2$ is not considered to be prescribed.

\includegraphics[width=\textwidth]{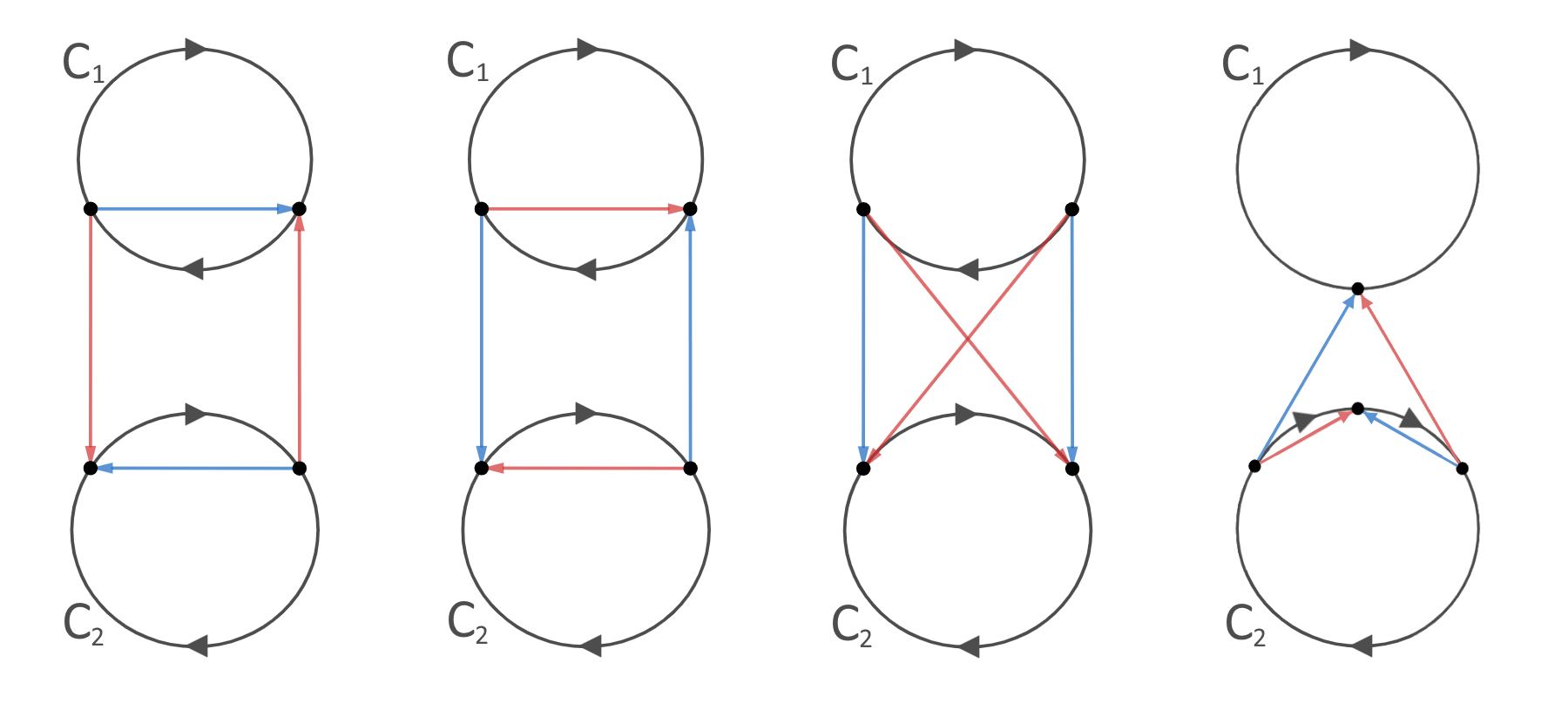}

    \textbf{Type 2(a):} Let $M$ be a matrix corresponding to a graph of type 2(a). Every walk from a vertex of $C_1$ to a vertex of $C_2$ involves an odd number of red arcs, and is therefore a negative walk. Because there is a walk of length $k$ from a vertex in $C_1$ to a vertex in $C_2$ for every $k \geq 1$, this means that there are negative entries in off-diagonal positions of $M^k$ for every $k \geq 1$. Thus matrices of type 2(a) cannot have finite multiplicative order.
    
    \textbf{Type 2(b):} Let $\Gamma$ be a $PT$-graph of type 2(b), where $m_1$ and $m_2$ are the lengths of the cycles $C_1$ and $C_2$ respectively. We may order the vertices of $\Gamma$ so that the corresponding matrix is 
    $$
    A = C_{m_1}\oplus C_{m_2} + T_{1,m_1+k_2; m_1+1,k_1},
    $$
    where $\oplus$ denotes the matrix direct sum. A routine calculation using row operations shows that the characteristic polynomial of the $A$ given by 
    \[p(x) = x^{m_1+m_2}+x^{m_1+m_2-k_1}+x^{m_1+m_2-k_2}-x^{m_1}-x^{m_2}-x^{m_1-k_1}-x^{m_2-k_2}+1\]
    
    If $A$ has finite order, then $p(x)$ must be either palindromic or skew-palindromic.
    \begin{itemize}
        \item If $p(x)$ is skew-palindromic, then the leading coefficient has opposite sign to the constant term. It follows that $x^{m_1-k_1} = x^{m_2-k_2} = x^0$. So $k_1 = m_1$ and $k_2 = m_2$.
        \item If $p(x)$ is palindromic, then  $(m_1+m_2-k_1) + (m_1+m_2-k_2) = m_1+m_2$, so $k_1+k_2 = m_1+m_2$. Because $k_1 \leq m_1$ and $k_2 \leq m_2$, it follows that $k_1=m_1$ and $k_2=m_2$ as above.
    \end{itemize}
    The conditions $k_1=m_1$ and $k_2=m_2$ hold only if the negative entries of the $T$-block in $A$ cancel positive entries of $C_{m_1}\oplus C_{m_2}$, which means that $A=C_{m_1+m_2}$ and in particular $A$ is a permutation matrix. Thus every $PT$-matrix of type 2(b) that has finite order is a permutation matrix.
    
    \textbf{Type 2(c)} and \textbf{Type 2(d)} are considered in Sections \ref{type 2c} and \ref{type 2d}; non-permutation examples of finite order occur in these cases.

\subsection{Type 3: three cycles}
The distinct $PT$-graph structures (up to arc reversal) where the vertices of $\Gamma_T$ belong to three cycles in $\Gamma_P$ are depicted below. We refer to these cases as  \emph{type 3(a)}, \emph{type 3(b)}, and \emph{type 3(c)}, respectively.

\includegraphics[width=0.95\textwidth]{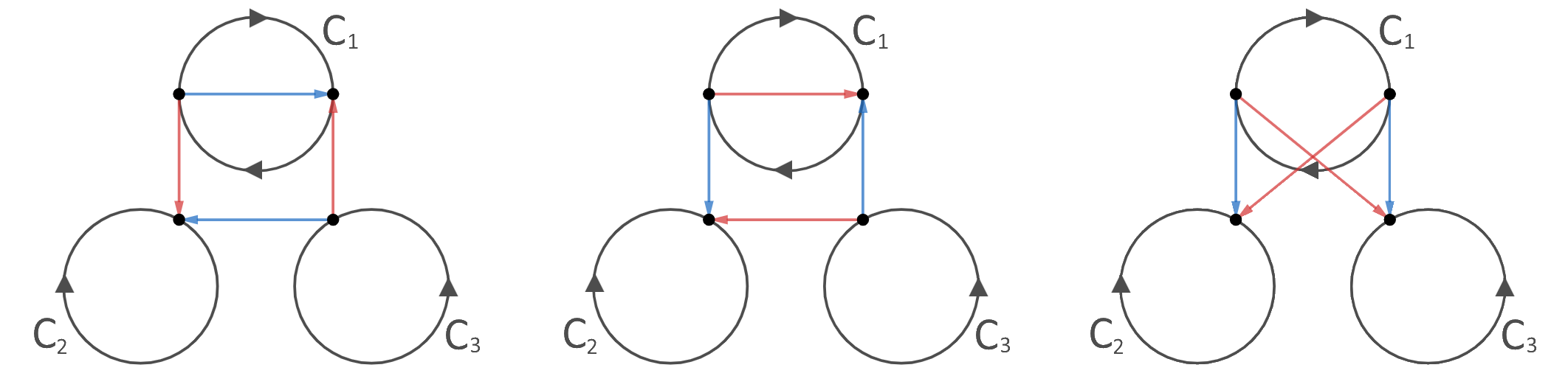}

    \textbf{Type 3(a):} In a $PT$-graph of type 3(a), every walk from a vertex in $C_1$ to a vertex in $C_2$ involves exactly one red arc, and is therefore a negative walk. Since there is a walk of length $k$ from a vertex of $C_1$ to a vertex of $C_2$ for every $k \geq 1$, it follows that negative entires occur in all positive powers of $PT$-matrices of type 3(a). Hence no $PT$-matrix of this type has finite order.

    \textbf{Type 3(b):} If the vertices of a $PT$-graph of type 3(b) are listed with those of the cycle $C_3$ first, followed by those of $C_1$ and then $C_2$, the corresponding $PT$-matrix is block upper-triangular with three square blocks on the diagonal. Subject to a suitable ordering of the vertices of $C_1$, the second diagonal block is the companion matrix of a polynomial of the form $x^m+x^k-1$, where $k<m$. Such a polynomial cannot be palindromic or skew-palindromic and hence cannot be a product of cyclotomic polynomials. Hence  a $PT$-matrix of type 3(b) cannot have finite multiplicative order.

Matrices of \textbf{type 3(c)} are considered in Section \ref{type 3c}.

\subsection{Type 4: four cycles} 
The case where the vertices of $\Gamma_T$ belong to four different cycles of $\Gamma_P$ is depicted below. 

\includegraphics[width=\textwidth]{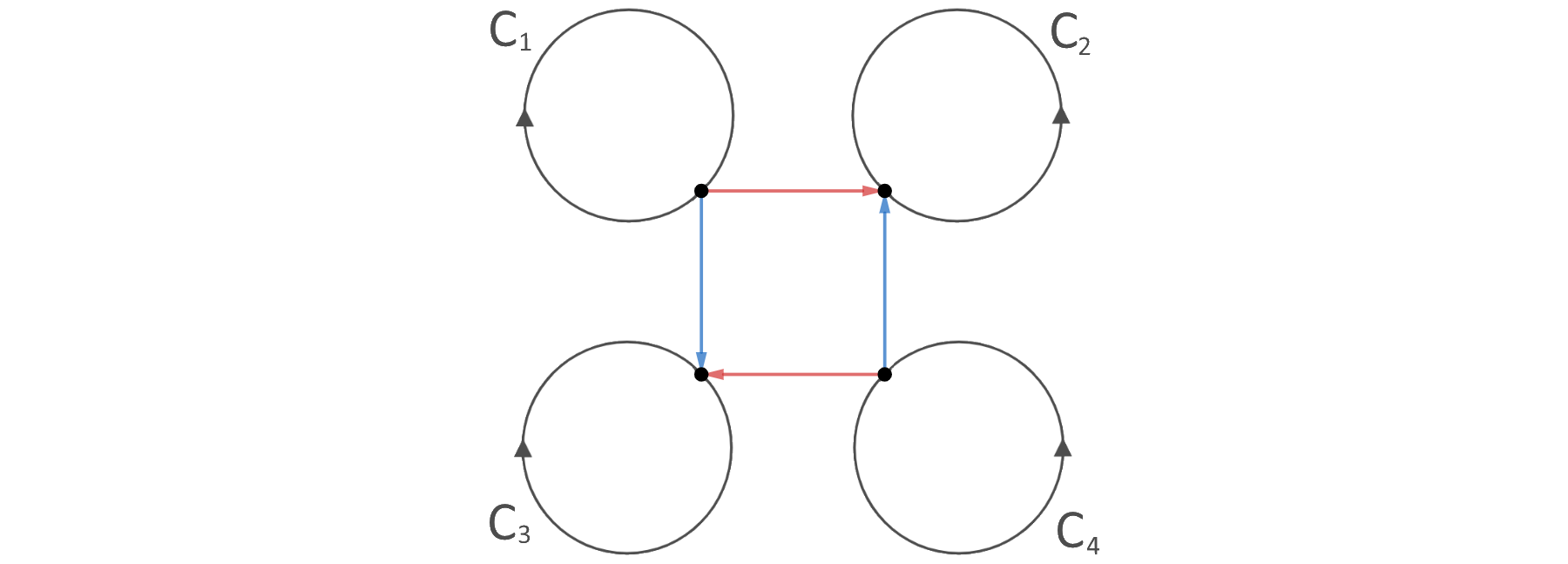}

Every walk from a vertex of $C_1$ in the above graph to a vertex of $C_3$ involves only blue arcs and is therefore positive. Such walks occur of all positive lengths, and so every power of a $PT$-matrix of type 4 has positive off-diagonal entries. Hence no $PT$-matrix of type 4 has finite multiplicative order.

We summarize the conclusions of Section 3 below. 

\begin{theorem}\label{graphs} 
Let $\Gamma$ be a weakly connected $PT$-graph with two red arcs $(u_1,v_1)$ and $(u_2,v_2)$, and with no multiple blue arcs. Let $\Gamma_T$ be the subgraph of $\Gamma$ with vertex set $\{u_1,v_1,u_2,v_2\}$, whose arc set includes the two red arcs and the blue arcs $(u_1,v_2)$ and $(u_2,v_1)$. Let $\Gamma_P$ be the subgraph of $\Gamma$ on the full vertex set, whose arcs are exactly those that do not belong to $\Gamma_T$. Then $\Gamma_P$ is composed of disjoint directed cycles. If the $(0,1,-1)$-matrices corresponding to $\Gamma$ have finite multiplicative order, then either $\Gamma$ or its reverse is of one of the following four types. 
\begin{itemize}
    \item Type 1: $\Gamma_P$ is a single directed cycle.
    \item Type 2(c): The graph $\Gamma_P$ consists of two disjoint directed cycles, with $u_1$ and $u_2$ belonging to one of these cycles and $v_1$ and $v_2$ to the other. The vertices $u_1,u_2,v_1,v_2$ are distinct. 
    \item Type 2(d): The graph $\Gamma_P$ consists of two disjoint directed cycles, with $u_1$ belonging to one of these cycles $u_2,v_1$ and $v_2$ to the other. The vertices $v_1$ and $u_2$ are distinct, but $v_2$ may coincide with one of these.
    \item Type 3(c): The graph $\Gamma_P$ consists of three disjoint directed cycles. The vertices  $u_1$ and $u_2$ belong to the same cycle of $\Gamma_P$, and the other two cycles each includes one of $v_1$ and $v_2$. The vertices $u_1,u_2,v_1,v_2$ are distinct in this case. 
\end{itemize}
\end{theorem}

\section{Elementary $PT$-matrices of finite multiplicative order}\label{polynomials}

We refer to a $PT$-matrix as \emph{elementary} if its graph is weakly connected. Every $PT$-matrix is permutation similar to the matrix direct sum of an elementary $PT$-matrix and a permutation matrix, so we focus on the elementary case.

In this section we analyse elementary $PT$-matrices of finite order, which correspond to graphs of one of the four types identified in Theorem \ref{graphs}. We establish a classification, up to permutation similarity and transposition, of elementary $PT$-matrices of finite multiplicative order. Examples of finite order exist in all four cases, but their orders can differ from those of permutation matrices of the same size only for Types 1 and 2(d). No generality is lost by restricting to $PT$-matrices with weakly connected graphs, since the addition of a new connected component to a graph is equivalent to extending the corresponding matrix via a direct sum. Central to our analysis is the fact that the characteristic polynomials of $PT$-matrices have a particularly amenable form in the cases of interest. This enables us to identify all $PT$-matrices whose characteristic polynomial is a product of cyclotomic polynomials. As we noted in Section \ref{algebra}, such a matrix has finite multiplicative order if and only if it is diagonalizable. Lemma \ref{diagonalizability} below is the main technical tool that we employ to determine necessary and sufficient conditions for diagonalizability of $PT$-matrices of the four types.

\begin{lemma}\label{diagonalizability}
Let $A$ be a block upper triangular matrix in $M_n(\Q )$, with diagonalizable square $p\times p$ and $q\times q$ blocks $P$ and $Q$ in the upper left and lower right respectively, where $p+q=n$. Let $g(x)$ be the greatest common divisor of the minimum polynomials of $P$ and $Q$ respectively, and let $N$ be the upper right $p\times q$ block of $g(A)$. Then $A$ is diagonalizable if and only if $Nv$ belongs to the columnspace of $g(P)$, for every vector $v$ in the right nullspace of $g(Q)$.
\end{lemma}

\begin{proof}
We write $m_P(x)$, $m_Q(x)$ and $m_A(x)$ respectively for the minimum polynomials of $P,\ Q$ and $A$. Since $P$ and $Q$ are diagonalizable, neither $m_P(x)$ nor $m_Q(x)$ has any repeated irreducible factors. We define the polynomials $p(x)$ and $q(x)$ by $m_P(x)=p(x)g(x)$ and $m_Q(x)=q(x)g(x)$. We note that $\gcd (p(x),q(x))=1$. Since $A$ is diagonalizable if and only if its minimum polynomial has distinct roots, and since the irreducible factors of $m_A(x)$ are exactly those of $m_P(x)$ and $m_Q(x)$, it follows that $A$ is diagonalizable if and only if $m_A(x)=p(x)q(x)g(x)$. 

We now consider under what conditions the matrix product $A'=p(A)q(A)g(A)$ is equal to zero. Since $P$ and $Q$ are diagonalizable, $\C^p$ and $\C^q$ have bases consisting of eigenvectors of $P$ and $Q$ respectively. Thus $\C^n$ has a basis $\{u_1,\dots ,u_p,v_1,\dots ,v_q\}$ where each $u_i$ is an eigenvector of $P$ with $q$ zeros appended, and each $v_i$ is an eigenvector of $Q$ with $p$ zeros prepended. Then $A'u_i=0$ for $1\le i\le p$, since $p(A)u_i=0$. 

The last $q$ entries of each the vectors $v_1,\dots ,v_q$ comprise an eigenvector of $Q$ whose corresponding eigenvalue is a root either of $q(x)$ or $g(x)$. If $v_i$ corresponds to a root of $q(x)$, then $q(A)v_i$ has zeros in its last $q$ positions and $A'v_i=m_P(A)q(A)v_i=0$.

Now let $v$ be a vector in $\{v_1,\dots ,v_q\}$ that corresponds to an eigenvalue of $Q$ that is a root of $g(x)$. Then 
$$
A'v  = q(A)p(A)g(A) = q(A)p(A)\left(\begin{array}{c}
Nv \\ 0_{q\times 1} \end{array}\right)=
\left(\begin{array}{c}
q(P)p(P)Nv \\ 0_{q\times 1} \end{array}\right).
$$
Since no root of $q(x)$ is an eigenvalue of $P$, the matrix $q(P)$ is nonsingular, and $A'v=0$ if and only if $p(P)Nv=0$; that is if and only if the vector $Nv$ belongs to the right nullspace of $p(P)$. Since the minimum polynomial of $P$ is $p(x)g(x)$ and $P$ is diagonalizable, the right nullspace of $p(P)$ is equal to the columnspace of $g(P)$. We conclude that $A$ is diagonalizable if and only if $Nv$ belongs to the columnspace of $g(P)$ for every vector $v$ in the right nullspace of $g(Q)$. 
\end{proof}

The condition of Lemma \ref{diagonalizability} is equivalent to the assertion that the zero eigenvalue of $g(A)$ has full geometric multiplicity, but our analysis will employ the formulation in the lemma. This depends on the feasibility of calculating the entries of $g(A)$ and the right nullspace of $g(Q)$. In most cases of interest, $Q$ is the matrix $C_q$, representing a single $q$-cycle, and $P$ is the companion matrix of a polynomial with few non-zero coefficients. The upper-right $p\times q$ block $M$ of $A$ is a sparse matrix of rank 1, with at most four non-zero entries spread over at most two columns. The polynomial $g(x)$ has the form $x^g\pm 1$ for some integer $g$, so $N$ is the upper right block of $A^g$. This is given by 
$$
P^{g-1}M+P^{g-2}MQ+P^{g-3}MQ^2+\dots +MQ^{g-1}.
$$
The effect of right multplication by $C_q$ on an entry of $M$ is to shift it one step left, or into Column $q$ if it is in Column 1. The effect of left multiplication by a companion matrix $P$ is to shift the entry one step downward, unless it is in Row $q$, in which case the final column of $P$ enters. In most cases of interest, an entry $a$ in position $(i,j)$ of $M$ leads to $g$ appearances of $a$ in $A^g$, in a diagonal pattern of positions starting at $(i,j-g+1)$ and proceeding downwards and to the right. 

We proceed to consider each of the four possible graph types listed in Theorem \ref{graphs}, where the above remarks will apply. 
For a positive integer $t$, we write $[t]_2$ for the highest power of $2$ that divides $t$. We note the following properties of common divisors of polynomials of the form $x^t\pm 1$. 

\begin{lemma}\label{gcd}
Let $s$ and $t$ be positive integers. Then
\begin{itemize}
    \item $\gcd (x^s-1,x^t-1) = x^{\gcd (s,t)}-1.$
    \item $\gcd (x^s+1,x^t+1)=\left\{\begin{array}{cl} x^{\gcd (s,t)}+1 & \mathrm{if}\ [s]_2=[t]_2 \\ 1 & \mathrm{if}\ [s]_2\neq [t]_2\end{array}\right.$
    \item $\gcd (x^s-1,x^t+1)=\left\{\begin{array}{cl}x^{\gcd (s,t)}+1 & \mathrm{if}\ [s]_2>[t]_2 \\ 1 & \mathrm{if}\ [s]_2\le [t]_2 \end{array}\right.$
    
\end{itemize}
\end{lemma}

\subsection{$PT$-matrices of Type 1: a single $n$-cycle}\label{type 1}
We may order the vertices of a $PT$-graph of Type 1 so that its corresponding $PT$-matrix has the form $C_n+T(1,j_1,d+1,j_2)$, where $d\le \frac{n}{2}$. We write $A$ for $C_n+T(1,j_1,d+1,j_2)$ and observe (using cofactor expansion on the first row), that the characteristic polynomial $p(x)$ of $A$ has the following simple form, with at most six non-zero terms. For an integer $t$, we write $[t]$ for the remainder on dividing $t$ by $n$. 

\begin{theorem}\label{char poly 1}
For $d\le\frac{n}{2}$, the characteristic polynomial of $A=C_n+T(1,j_1,d+1,j_2)$ is
  $$
p(x)=x^n-x^{n-j_1}+x^{[n-j_1+d]}+x^{n-j_2}-x^{[n-j_2+d]}-1.
$$
\end{theorem}

We wish to determine when $p(x)$ is a product of cyclotomic polynomials. We begin by considering its constant term, which may differ from $-1$ only if either $j_1$ or $j_2$ is equal to either $n$ or $d$. The only such case in which the constant term is $1$ or $-1$ is when $j_1=d$ and $j_2=n$. In this case $A$ is a permutation matrix, corresponding to a pair of cycles of lengths $d$ and $n-d$. 
 
We assume now that neither $j_1$ nor $j_2$ is equal to $n$ or $d$. Since its constant term is $-1$, the polynomial $p(x)$ of Theorem \ref{char poly 1} can be a product of cyclotomic polynomials only if it is skew-palindromic. This occurs in the following two cases. 
\begin{itemize}
    \item[\textbf{Case 1}] 
    $n-j_1+n-j_2=n$ and $[n-j_1+d]+[n-j_2+d]=n$. \\
    From the first equation, $j_1+j_2=n$, so the second equation reduces to $[j_2+d]=[j_1+d]=n$, which means that $j_1+j_2+2d$ is a multiple of $n$. Since $j_1+j_2=n$ and $d\le\frac{n}{2}$, this can be satisfied only if $2d=n$. We note that $|j_1-j_2|$, which is the distance between the columns occupied by entries of $T$, is even in this situation. \\
    In this case, $\Gamma (A)$ consists of a directed cycle of length $n=2d$, which we write as
    $$
    v_1\to v_n\to v_{n-1}\to\dots \to v_{2}\to v_1,
    $$
    with additional blue and red arcs from $v_1$ and $v_{\frac{n}{2}+1}$ to $v_{j_1}$ and $v_{j_2}$, corresponding to the entries of $T$, where $j_1\neq j_2$ and $j_1-j_2$ is even. It remains to identify the values of $j_1$ and $j_2$ for which $A$ has finite order. 
    \item[\textbf{Case 2}]
    $n-j_1+[n-j_1+d]=n$ and $n-j_2+[n-j_2+d]=n$. \\
    This occurs only if $\{n+d-2j_1,n+d-2j_2\}=\{0,n\}$. This means that $d$ and $n$ are both even and $\{j_1,j_2\}=\{\frac{d}{2},\frac{n+d}{2}\}$. In particular, $|j_1-j_2|=\frac{n}{2}$. \\
    In this situation we may label the vertices so that the graph $\Gamma(A)$ consists of a directed $n$-cycle on blue arcs as in Case 1 above, with four additional arcs corresponding to the entries of $T$, directed from $v_1$ and $v_{d+1}$ to some $v_j$ and $v_{j+\frac{n}{2}}$, where $d$ is even. 
    \end{itemize}
    Reversing all arcs in a graph arising in Case 2 results in a graph of the type described in Case 1. It follows that every $PT$-matrix arising in Case 2 above is permutation similar to the transpose of one that arises in Case 1. For this reason, we consider Case 1 only. 

We now write $j$ for the minimum of $j_1$ and $j_2$, and write $n=2d$. 
\begin{itemize}
    \item If $j_1<j_2$, then $j=j_1$, $A=C_n+T(1,j,d+1,n-j)$, and $$p(x)= 
x^n-x^{n-j}-x^{d+j}+x^{d-j}+x^j-1 = (x^j-1)(x^{d-j}-1)(x^d-1).$$
\item
If $j_1>j_2$, then $j=j_2$, $A=C+T(1,n-j,d+1,j)$, and 
$$
p(x)=x^n+x^{n-j}+x^{d+j}-x^{d-j}-x^j-1 =(x^j+1)(x^{d-j}+1)(x^d-1).
$$
\end{itemize}

In any case where  $p(x)$ is a product of distinct cyclotomic factors, we can conclude that the matrix has finite multplicative order. In the case where $p(x)$ is a product of cyclotomic factors with repetition, we need to consider the relationship between $p(x)$ and the minimum polynomial $m(x)$ of $A$. To this end we consider the minimal $A$-invariant subspace of $\C^n$ that contains the vector $v_1$, which has 1 in position 1, $-1$ in position $d+1$, and zeros elsewhere. This vector spans the $1$-dimensional column space of $T$. 

For $i=1,\dots ,d$, we write $v_i$ for the vector in $\C^{n}$ that has $1$ in position $i$, $-1$ in position $d+i$ and zeros elsewhere. We write $V$ for the span of the $v_i$ which clearly has dimension $d$ and consists of all vectors in $\C^n$ of the form $\left(\begin{array}{r} v \\ -v\end{array}\right)$, where $v\in\C^d$. It is evident that $V$ is $A$-invariant, since it is $C_n$-invariant, and $Tx\in\langle v_1\rangle\subseteq V$ for all $x\in\C^n$. Moreover, $C_n^i v_1=v_{i+1}$, for $i=1,\dots ,d-1$, and $C_n^d v_1=-v_1$. We note that $Av_1=v_2+\alpha v_1$, where $\alpha\in\{-1,0,1\}$. Applying $A$ repeatedly, it follows for $i\le d-1$ that $A^iv_1=v_{i+1}+w$, where $w$ is a linear combination of $v_1,\dots ,v_i$. In particular, $\Bcal_1=\{v_1,Av_1,A^2v_1,\dots ,A^{d-1}v_1\}$ is a linearly independent set and a basis of $V$. It follows that the restriction to $V$ of the linear transformation determined by $A$ is non-derogatory; its minimum polynomial has degree $d$. We extend $\Bcal_1$ to a basis $\Bcal$ of $\C^n$ by appending the standard basis vectors $e_{d+1},\dots ,e_n$. 
Rewriting with respect to the basis $\Bcal$, we find that $A$ is similar to the matrix $A'$ with the following block upper triangular form.
\begin{itemize} 
\item
The lower right $d\times d$ block of $A'$ is $C_d$, the companion matrix of $x^d-1$. That $Ae_i\in e_{i+1}+V$ is clear for $i=d+1,\dots ,n-1$, and $Ae_n=e_1=v_1+e_{d+1}$, since the last column of $A$ is just $e_1$.
\item 
The upper right $d\times d$ block of $A'$ has only zero entries outside its first row. In the first row, the entry in the $(1,n-j)$ position of $A'$ is $1$ or $-1$ (according as $A_{1,n-j}$ is positive or negative), and the entry in the $(1,n)$-position is $1$. All other entries in this region are zeros.
\item
The lower left $d\times d$ block of $A'$ is $0_{d\times d}$.
\item
The upper left $d\times d$ block of $A'$ is the companion matrix of the minimum polynomial of the restriction to $V$ of the linear transformation determined by $A$. This is $\frac{p(x)}{x^d-1}$, which is $(x^j-1)(x^{d-j}-1)$ if $A_{1j}=1$, or $(x^j+1)(x^{d-j}+1)$ if $A_{1j}=-1$.
\end{itemize}
Since $(x^j-1)(x^{d-j}-1)$ has $1$ as a repeated root, the upper left block of $A'$ can have finite order only if $A_{1j}=-1$. In this situation, the block has finite order if and only if the polynomials $x^j+1$ and $x^{d-j}+1$ are relatively prime, which occurs if and only if $[j]_2\neq [d-j]_2$, as noted in Lemma \ref{gcd}. We assume that this condition holds, so that the upper left block of $A'$ has finite order. 

Since $[j]_2\neq [d-j]_2$, it follows that $[d]_2 = \min ([j]_2,[d-j]_2)$ and hence that $x^d-1$ is relatively prime to both $x^j+1$ and $x^{d-j}+1$. Thus $A=C_{2d}+T(1,2d-j,d+1,j)$ has finite multiplicative order in any case where $j<d$ and $[j]_2\neq [d-j]_2$. Since the minimum polynomial of $A$ in this situation is $(x^j+1)(x^{d-j}+1)(x^d-1)$, the order is $\lcm(2j,2d-2j,d)$.

For $PT$-graphs and $PT$-matrices of Type 1, we have the following conclusions. The $PT$-graph $\Gamma$ is defined here as in Theorem \ref{graphs}. For vertices $u$ and $v$ of $\Gamma$, $d_P(u,v)$ denotes the length of the path from $u$ to $v$ along the cycle $\Gamma_P$. 

\begin{theorem}\label{Type 1 graph} 
  Let $\Gamma$ be a $PT$-graph of type 1 of order $n$, with red arcs $(u_1,v_1)$ and $(u_2,v_2)$. Then the $(0,1,-1)$-matrix corresponding to $\Gamma$ (with respect to a vertex ordering) has finite multiplicative order if and only if $n=2d$ is even, the 2-parts of the integers $d_P(v_1,u_1)+1$ and $d_P(u_2,v_1)-1$ are different, $d_P(v_1,u_1)+d_P(v_2,u_2)=d-2$, and either $d_P(u_1,u_2)=d$ or $d_P(v_1,v_2)=d$. 
\end{theorem}
The two versions of the final condition in the statement of Theorem \ref{Type 1 graph} correspond to the cases where $\Gamma$ itself, or its reverse, is described by a matrix having the form in the above discussion.

\begin{theorem}\label{Type 1 matrix} 
A $n\times n$ $PT$-matrix $A$ of Type 1 has finite multiplicative order if and only if $n$ is even and either $A$ or its transpose is permutation similar to a matrix of the form 
$$
C_n+T(1,n-j,d+1,j),
$$
where $n=2d,\ j<d$ and $[j]_2\neq [d-j]_2$. In this case the multiplicative order of $A$ is $\lcm (2j,2d-2j,d)$. 
\end{theorem}

It is possible for a $n\times n$ $PT$-matrix of Type 1 to have a multiplicative order that does not occur as the order of a permutation matrix in $\S_n$. For example if $n=10$, choosing $j=1$ or $j=2$ gives $PT$-matrices (shown below) whose respective multiplicative orders are $\lcm (2,8,5)=40$ and $\lcm (4,6,5)=60$. Neither 40 nor 60 occurs as the order of an element in the symmetric group $S_{10}$, since neither occurs as the least common multiple of the integers in a partition of 10. 
$$
\begin{array}{c}
\left(\begin{array}{rrrrrrrrrr}
-1 & 0 & 0 & 0 & 0 & 0 & 0 & 0 & 1 & 1 \\
 1 & 0 & 0 & 0 & 0 & 0 & 0 & 0 & 0 & 0 \\ 
 0 & 1 & 0 & 0 & 0 & 0 & 0 & 0 & 0 & 0 \\
 0 & 0 & 1 & 0 & 0 & 0 & 0 & 0 & 0 & 0 \\
 0 & 0 & 0 & 1 & 0 & 0 & 0 & 0 & 0 & 0 \\
 1 & 0 & 0 & 0 & 1 & 0 & 0 & 0 & -1 & 0 \\
 0 & 0 & 0 & 0 & 0 & 1 & 0 & 0 & 0 & 0 \\
 0 & 0 & 0 & 0 & 0 & 0 & 1 & 0 & 0 & 0 \\
 0 & 0 & 0 & 0 & 0 & 0 & 0 & 1 & 0 & 0 \\
 0 & 0 & 0 & 0 & 0 & 0 & 0 & 0 & 1 & 0 \\
\end{array}\right) \\
 \ \\
C_{10}+T(1,9,6,1), \ \mathrm{order}\ 40
\end{array} \ \ \ \ 
\begin{array}{c}
\left(\begin{array}{rrrrrrrrrr}
 0 & -1 & 0 & 0 & 0 & 0 & 0 & 1 & 0 & 1 \\
 1 & 0 & 0 & 0 & 0 & 0 & 0 & 0 & 0 & 0 \\ 
 0 & 1 & 0 & 0 & 0 & 0 & 0 & 0 & 0 & 0 \\
 0 & 0 & 1 & 0 & 0 & 0 & 0 & 0 & 0 & 0 \\
 0 & 0 & 0 & 1 & 0 & 0 & 0 & 0 & 0 & 0 \\
 0 & 1 & 0 & 0 & 1 & 0 & 0 & -1 & 0 & 0 \\
 0 & 0 & 0 & 0 & 0 & 1 & 0 & 0 & 0 & 0 \\
 0 & 0 & 0 & 0 & 0 & 0 & 1 & 0 & 0 & 0 \\
 0 & 0 & 0 & 0 & 0 & 0 & 0 & 1 & 0 & 0 \\
 0 & 0 & 0 & 0 & 0 & 0 & 0 & 0 & 1 & 0 \\
\end{array}\right) \\
 \ \\
C_{10}+T(1,8,6,2), \ \mathrm{order}\ 60
\end{array}
$$
It will be shown in Section \ref{ASM section} that the $PT$-matrices arising in Theorem \ref{Type 1 matrix} are all permutation similar to alternating sign matrices.

\subsection{Type 2(c): two cycles connected by four arcs of $T$}\label{type 2c}

In a $PT$-graph of Type 2(c) in Theorem \ref{graphs}, the permutation component consists of a pair of cycles of lengths $p$ and $q$ respectively, where $p+q=n$ and each of $p$ and $q$ is at least 2. The $T$-component  contributes four additional arcs, a pair of red arcs directed from distinct vertices $x_1$ and $x_2$ of the $p$-cycle to distinct vertices $y_1$ and $y_2$ respectively of the $q$-cycle, and a pair of blue arcs from $x_1$ to $y_2$ and from $x_2$ to $y_1$. We assume that $y_1$ and $y_2$ are labelled so that the directed path along the $q$-cycle from $y_2$ to $y_1$ is no longer that the one from $y_1$ to $y_2$. 

We order the vertices of $\Gamma$ as follows. We begin with the vertices of the $p$-cycle, starting with $x_1$ and proceeding against the direction of the arcs in $C$. We continue with the vertices of the $q$-cycle, proceeding against the direction of the arcs of the cycle, to end with $y_2$.

With respect to this ordering, the $n\times n$ matrix $A$ of $\Gamma$ has the following description, where $\oplus$ denotes the matrix direct sum.
$$
A = (C_p\oplus C_q)+T(1,n,h+1,n-l),
$$
where $h$ is the length of the path from $x_2$ to $x_1$ in the $p$-cycle, and $l$ is the length of the path from $y_2$ to $y_1$ in the $q$-cycle. We derive conditions on $h,l,p$ and $q$, for $A$ to have finite order. Since $A$ is block upper triangular with $C_p$ and $C_q$ as its diagonal blocks, it has finite order if and only if its minimum polynomial is $\lcm (x^p-1,x^q-1)$, and in this case its order is $\lcm (p,q)$. If this occurs, then neither $p$ nor $q$ can be a divisor of the other, since inspection of walks of length $\max (p,q)$ from vertices of $C$ to vertices of $C'$ shows that $A^{\max(p,q)}$ has non-zero entries in its upper right $p\times q$ region. For example, there is at least one positive $q$-walk from $u$ to the in-neighbour of $v'$ in $C'$, and no negative one. There is at least one positive $p$-walk from the out-neighbour of $v$ in $C$ to $u'$, and no negative one. Thus it is not possible that $A$ has order $p$ or $q$, and we may confine our attention to cases where neither $p$ nor $q$ divides the other.

We write $g$ for $\gcd (p,q)$ and note that $q\le\frac{q}{2}$, and so $n-l\ge p+g$, since $l\le\frac{q}{2}$. We apply Lemma \ref{diagonalizability} with $g(x)=\gcd (x^p-1,x^q-1)=x^g-1$. Then
\begin{equation}\label{A^g}
g(A)=A^g-I_n = (g(C_p)\oplus g(C_q))+\sum_{i=1}^g T(i,[h+i],n-g+i,n-l-g+i),
\end{equation}
where $[h+i]=h+i-p$ if $h+i$ exceeds $p$, and is otherwise equal to $h+i$. In place of the 1 in the upper right position of $A$, $A^g$ (or $A^g-I_n$) has a strip of $g$ entries equal to 1, in a diagonal arrangement from position $(1,n-g+1)$ to position $(g,n)$. A similar pattern occurs for each of the three other non-zero entries in the upper right $p\times q$ block of $A$.

We write $N$ for the upper right block of $g(A)$. By Lemma \ref{diagonalizability}, we need to consider whether $Nv$ belongs to the column space $U$ of $g(C_p)$, for every vector $v$ satisfying $g(C_q)v=0$. The column space of $C_p^g-I_p$ consists of all vectors $u\in\C^p$ for which the sum of the $\frac{p}{g}$ entries $u_i,u_{i+g},\dots ,u_{i+(\frac{p}{g}-1)g}$ is zero, for $i=1,\dots g-1$. 

A basis for the right nullspace of $g(C_q)$ is given by $\{w_1,\dots ,w_g\}$, where $w_i$ has entries equal to 1 in the $\frac{q}{g}$ positions with indices congruent to $i$ modulo $g$, and zeros elsewhere.

According to (\ref{A^g}), the non-zero columns of $N$ occur in two (possibly overlapping) contiguous bands, from Column $q-g+1$ to Column $n$, and from Column $q-l-g+1$ to Column $q-l$ of $N$. The nonzero entries of $N$ occur as follows, where $1\le i\le g$. 
\begin{itemize}
\item
Column $q-l-g+i$ of $N$ has $-1$ in position $i$ and 1 in position $[h+i]$. 
\item 
Column $q-g+i$ has $1$ in position $i$ and $-1$ in position $[h+i]$. 
\end{itemize}
It follows that $Nv_i$ is either equal to Column $q-g+i$ of $N$ (if $i\le g-l$) or to the sum of Columns $q-g+i$ and $q-l-g+\overline{l+i}$ of $N$, where $\overline{l+i}$ is the reminder on dividing $l+i$ by $g$. This sum is zero if $g|l$, otherwise it is the vector with $1$ in positions $i$ and $[\overline{l+i}+h]$, $-1$ in positions $\overline{l+i}$ and $[i+h]$, and zeros elsewhere. This vector belongs to the columnspace of $g(P)$ only if $g|h$. It now follows from Lemma \ref{diagonalizability} that $A$ has finite multiplicative order if and only if $g$ divides either $l$ or $h$, giving the following conclusions for $PT$-graphs and $PT$-matrices of type 2(c), as described in Theorem \ref{graphs}. 

\begin{theorem}\label{2c graph}
Let $\Gamma$ be a $PT$-graph of type 2(c), in which the two cycles have lengths $p$ and $q$, the vertices $u_1$ and $u_2$ belong to the $p$-cycle, and the vertices $v_1$ and $v_2$ belong to the $q$-cycle. Then a $(0,1,-1)$-matrix corresponding to $\Gamma$ has finite multiplicative order if and only if at least one of $d_P(u_1,u_2)$ and $d_P(v_1,v_2)$ is a multiple of $\gcd (p,q)$. 
\end{theorem}

\begin{theorem}\label{2c matrix}
Let $A$ be a $PT$-matrix of type 2(c). Then $A$ has finite multiplicative order if and only if $A$ or its transpose is permutation similar to the matrix $C_p\oplus C_q+T(1,n,h+1,n-l)$, where $1\le h< p,\ 1\le l<q$, and $\gcd (p,q)$ divides at least one of $h$ and $l$. When this occurs, the order of $A$ is $\lcm (p,q)$.
\end{theorem}

Every matrix arising in Theorem \ref{2c matrix} is similar to a permutation matrix. Those that are permutation similar to alternating sign matrices will be identified in Section \ref{ASM section}.

\subsection{Type 2(d): two cycles connected by two arcs of $T$}\label{type 2d}
We now consider $PT$-matrices corresponding to graphs of Type 2(d), in the classification given in Theorem \ref{graphs}. The matrix of Example \ref{finite} is of this type, with the underlying permutation involving a $4$-cycle and a fixed point. Let $\Gamma$ be a $PT$-graph of order $n=p+q$, consisting of disjoint directed cycles of lengths $p$ and $q$, whose arcs are coloured blue, and the following four additional arcs involving vertices $x_1,x_2,y_1$ of the $p$-cycle (with $x_1\neq x_2$) and a vertex $y_2$ of the $q$-cycle: $(x_1,y_1)$ and $(x_2,y_2)$, both coloured red, and blue arcs $(x_1,y_2)$ and $(x_2,y_1)$. 

We order the vertices of $\Gamma$ as follows. Vertices of the $p$-cycle are listed first, ordered against the direction of the arcs of the cycle, and ending with $y_1$. Vertices of the $q$-cycle follow, again against the direction of the arcs, and with $y_2$ appearing last. 

With respect to this ordering, the matrix $A$ of $\Gamma$ has the form 
$$
A=(C_p\oplus C_q)+T(i_1,n,i_2,p),
$$
where $i_1$ and $i_2$ are the (distinct) respective positions of $x_1$ and $x_2$ in the vertex ordering. We write $P$ for the upper left $p\times p$ block of $A$ and note that $P$ is the companion matrix of the polynomial 
$$
p(x)=x^p+x^{i_1-1}-x^{i_2-1}-1.
$$
Since $A$ is block upper triangular with $P$ as its upper-left block, $A$ may have finite multiplicative order only if $P$ does. The polynomial $p(x)$ cannot be palindromic, since its leading and constant coefficients cannot coincide. It is skew-palindromic only if $i_1+i_2=p+2$. We assume that this holds and rewrite $i_2$ as $i$. Then 
$$
p(x)=(x^{p-i+1}-1)(x^{i-1}+1),
$$
and, by Lemma \ref{gcd}, $P$ has finite order if and only if $[i-1]_2\ge [p-i+1]_2$. We assume that this condition holds, and hence that the order of $P$ is the least common multiple of $p-i+1$ and $2(i-1)$. We proceed to consider when $A$ has finite order. 

First we consider the case where $[q]_2>[i-1]_2$. We write $d$ and $g$ respectively for $\gcd (\frac{q}{2},i-1)$ and $\gcd (q,p-i+1)$. Then the greatest common divisor of the minimum  polynomials of $P$ and $C_q$ is $m(x)=(x^d+1)(x^g-1)$. 

By Lemma \ref{diagonalizability}, a necessary and sufficient condition for $A$ to be diagonalizable, or equivalently to have finite multiplicative order, is that the zero eigenvalue has full geometric multiplicity in the matrix $(A^d+I_n)(A^g-I_n)$. Since the polynomials $x^d+1$ and $x^g-1$ are relatively prime, this condition holds if and only if it holds separately for $A^d+I_n$  and $A^g-I_n$. We consider $A^d+I_n$ first. The right nullspace of $C_q^d+I_q$ has dimension $d$ and is spanned by the vectors $u_1,\dots ,u_d$, where for $i=1,\dots ,d$, $u_i$ has entries alternating between $1$ and $-1$ in positions $i,\ i+d,\ i+2d$ to $q-d+i$, starting with $1$ in position $i$. The columnspace of $P^d+I_p$ has dimension $p-d$, and is spanned by the vectors $w_1,\dots ,w_{p-d}$ in $\C^p$, where $w_i$ has 1 in positions $i$ and $i+d$, and zeros elsewhere. This space consists of all vectors in $\C^p$ with the property that for $i=1,\dots ,d$, the alternating sum of the sequence of entries in positions congruent to $i$ modulo $d$ is zero. 

The non-zero entries of the matrix $N_d$ are confined to the last $d$ columns. Since $u_i$ has $-1$ in position $q-d+i$ and its last $d$ entries are otherwise zero, the vectors $N_du_i$ are respective scalar multiples of the last $d$ columns of $N_d$. Thus $A^d+I_n$ satisfies the condition of Lemma \ref{diagonalizability} if and only if the columnspace of $N_d$ is contained in that of $P^d+I_p$. Column $q-d+1$ of $N_d$ is equal to column $q$ of the upper right block of the original $A$. It has two non zero entries; $-1$ and $1$ in positions $i$ and $p+2-i$. This vector $v$  belongs to the columnspace of $P^d+I$ if and only if $d$ divides $p+2-2i$ and $(p+2-2i)/d$ is even. This condition is equivalent to the statement that $d|p$ and $p/d$ is even, since $d$ divides $i-1$. Every column of $N_d$ is a linear combination of vectors with two non-zero entries of opposite sign, whose positions are separated either by $|p+2-2i|$ or by $2i-2$. Thus the condition that $p/d$ is an even integer, which is necessary to ensure that $v$ belongs to the columnspace of $P^d+I_p$, is also sufficient to ensure that $N_du$ belongs to this space, for every $u$ in the right nullspace of $Q^d+I$. 

We now consider the corresponding question for the matrix $A^g-I_n$. The analysis and conclusion here closely mirror those of Section \ref{type 2c}. The column space of $P^g-I_p$ has dimension $p$ and consists of all vectors in $\C^p$ with the property that for each $i\in\{1,\dots ,g\}$ the sum of all entries in positions congruent to $i$ modulo $g$ is zero. If $N_g$ is the upper right $p\times q$ block of $A^g-I_n$, then every column of $N_g$ occurs as the product $N_g u$ for some $u$ with $(Q^g-I_q)u=0$. As above, Column $q-g+1$ of $N_g$ has only two non-zero entries, of opposite sign and separated by a vertical distance of $|p+2-2i|$. This vector belongs to the column space of $P^g-I_p$ only if $g$ divides $p+2-2i$, and as above this condition is sufficient to ensure that the column space of $N_g$ is contained in that of $P^g-I_p$. Since $g$ is a divisor of $p-i+1$, the condition $g|p-2i+2$ is equivalent to $g|i-1$ and hence to $g|p$. 

\begin{theorem}\label{2d}
The $PT$-matrix $A=(C_p\oplus C_q)+T(i_1,n,i_2,p)$, where $i_1,i_2\le p$, has finite multiplicative order if and only if the following conditions are satisfied.
\begin{itemize}
    \item $i_1+i_2=p+2$, and $[i_2-1]_2\ge [p-i_2+1]_2$;
    \item $g|p$, where $g=\gcd (q,p-i_2+1)$;
    \item $[q]_2\le [i_2-1]_2$, \emph{or} $[q]_2>[i_2-1]_2$ and $p/d$ is an even integer, where $d=\gcd (\frac{q}{2},i_2-1)$.
\end{itemize}
\end{theorem}
If the conditions in Theorem \ref{2d} hold, then the order of $A$ is 
$$
\lcm (p-i_2+1,2i_2-2,q). 
$$
If $[q]_2\ge [i_2-1]_2$, this is equal to $\lcm (p-i_2+1,i_2-1,q)$, which is the order of a permutation of degree $p+q$. In general however, the order of $A$ need not coincide with that of a permutation. For example, we may set 
$p=i_2=2^k+1$, and $q=1$, for any positive integer $k$. We obtain a $(2^k+2)\times (2^k+2)$ matrix of order $\lcm (1,2^{k+1},1)=2^{k+1}$. An $n\times n$ permutation matrix of order $2^{k+1}$ exists only if $n\ge 2^{k+1}$. 

The $10\times 10$ example with $k=3$ and order 16 is below, along with an ASM to which it is permutation similar.

$$
\left(\begin{array}{rrrrrrrrrr}
 0 & 0 & 0 & 0 & 0 & 0 & 0 & 0 & 1 & 0 \\
 1 & 0 & 0 & 0 & 0 & 0 & 0 & 0 & -1 & 1 \\ 
 0 & 1 & 0 & 0 & 0 & 0 & 0 & 0 & 0 & 0 \\
 0 & 0 & 1 & 0 & 0 & 0 & 0 & 0 & 0 & 0 \\
 0 & 0 & 0 & 1 & 0 & 0 & 0 & 0 & 0 & 0 \\
 0 & 0 & 0 & 0 & 1 & 0 & 0 & 0 & 0 & 0 \\
 0 & 0 & 0 & 0 & 0 & 1 & 0 & 0 & 0 & 0 \\
 0 & 0 & 0 & 0 & 0 & 0 & 1 & 0 & 0 & 0 \\
 0 & 0 & 0 & 0 & 0 & 0 & 0 & 1 & 1 & -1 \\
 0 & 0 & 0 & 0 & 0 & 0 & 0 & 0 & 0 & 1\\
\end{array}\right) \\
 \ \ \ 
\left(\begin{array}{rrrrrrrrrr}
 0 & 0 & 0 & 0 & 0 & 0 & 0 & 1 & 0 & 0\\
 1 & 0 & 0 & 0 & 0 & 0 & 0 & -1 & 1 & 0\\ 
 0 & 1 & 0 & 0 & 0 & 0 & 0 & 0 & 0 & 0\\
 0 & 0 & 1 & 0 & 0 & 0 & 0 & 0 & 0 & 0\\
 0 & 0 & 0 & 1 & 0 & 0 & 0 & 0 & 0 & 0\\
 0 & 0 & 0 & 0 & 1 & 0 & 0 & 0 & 0 & 0\\
 0 & 0 & 0 & 0 & 0 & 1 & 0 & 0 & 0 & 0\\
 0 & 0 & 0 & 0 & 0 & 0 & 0 & 1 & -1 & 1\\
 0 & 0 & 0 & 0 & 0 & 0 & 0 & 0 & 1 & 0\\
 0 & 0 & 0 & 0 & 0 & 0 & 1 & 0 & 0 & 0\\
\end{array}\right) \\
$$

\subsection{Type 3(c): three cycles}\label{type 3c}
The remaining case concerns $PT$-matrices corresponding to graphs of type 3(c) in Corollary \ref{graphs}. The analysis for this resembles that of type 2(c), although it is simpler. A graph of Type 3(c) has three cycles of lengths $p, q$ and $m$. Additionally, it has a pair of blue arcs directed from distinct vertices $x_1$ and $x_2$ of the first cycle, respectively to vertices $y$ and $z$ of the second and third cycles, and a pair of red arcs from $x_1$ to $z$ and $x_2$ to $y$. With respect to a suitable ordering of the vertices, the corresponding matrix is
$$
A = (C_p \oplus C_q\oplus C_m)+T(1,p+q,i,p+q+m),
$$
where $1<i\le p$. Applying Lemma \ref{diagonalizability} as in previous cases, we find that the upper left $(p+q)\times (p+q)$ block $A_{p+q}$ of $A$ has finite multiplicative order if and only if $\gcd(p,q)$ divides $i-1$. We assume that this holds, and note that $A_{p+q}$ is then similar to $C_p\oplus C_q$, via a change of basis that does not affect the first $p$ basis elements. It follows that  $A$ itself is similar to the matrix $A'=(C_p\oplus C_q\oplus C_m)-E_{1,p+q+m}+E_{i,p+q+m}$, where $E_{i,j}$ has 1 in the $(i,j)$-position and zeros elsewhere. Now $A'$ has finite order if and only if the $(p+m)\times (p+m)$ matrix 
$(C_p\oplus C_m)-E_{1,p+m}+E_{i,p+m}$ does, and applying Lemma \ref{diagonalizability} confirms that this occurs if and only if $\gcd (p,m)$ divides $i-1$. 

\begin{theorem}\label{3c matrix}
Let $p,q,m$ be positive integers, with $p\ge 2$, and let $i$ be an integer with $1<i\le p$. The matrix 
$$
A = (C_p \oplus C_q\oplus C_m)+T(1,p+q,i,p+q+m),
$$
which represents a $PT$-graph of Type 3(c), has finite multiplicative order if and only if $\gcd (p,q)$ and $\gcd (p,m)$ both divide $i-1$. 
In this case the order of $A$ is $\lcm (p,q,m)$ and $A$ is similar to the permutation matrix $C_p\oplus C_q\oplus C_m$. 
\end{theorem}

\section{ASM-permutability}\label{ASM section}
It remains to determine which of the $PT$-matrices of finite multiplicative order are permutation similar to alternating sign matrices, or \emph{ASM-permutable}. We consider this question separately for the four types, using the following strategy in all cases. In Section 4, we identified a standard form for a $PT$-matrix of each of the four types, selected with ease of calculation of characteristic and minimum polynomials in mind. This amounts to a choice of ordering of the vertices of the corresponding digraph, which we now label as $1,2,\dots n$. We need to determine whether the same $n$ vertices can be rearranged to an \emph{ASM-ordering}, which means that the corresponding $(0,1,-1)$-matrix is an ASM. An ASM-ordering must satisfy four constraints, one arising from each of the four rows and columns in which the matrix has three non-zero entries, which are the rows and columns occupied by entries of the $T$-block. Each of the four constraints stipulates that a particular vertex, labelling the position of the $-1$ in the relevant row or column, must occur between two other vertices, which label the positive entries in the same row or column. The consistency of the four constraints needs to be checked.

\subsection{Type 1}
By Theorem \ref{Type 1 matrix}, every $PT$-matrix of finite order of Type 1 (or its transpose) is permutation similar to a matrix of the form 
$$
A=C_n+T(1,n-j,d+1,j),
$$
where $n=2d$ is even, $1\le j<d$ and $[j]_2\neq [d-j]_2$, which implies that $d\ge 3$. We write $\Gamma$ for the graph determined by the above matrix $A$, and write $1,2,\dots ,n$, for the vertices of $\Gamma$, in the order determined by $A$. An ordering of the vertices of $\Gamma$ is an ASM-ordering if and only if it satisfies the following conditions, determined respectively by Rows $1$ and $d+1$ of $A$, and by Columns $j$ and $n-j$.
\begin{enumerate}
    \item $j$ occurs between $n-j$ and $n$;
    \item $n-j$ occurs between $j$ and $d$;
    \item $1$ occurs between $d+1$ and $j+1$;
    \item $d+1$ occurs between $1$ and $n-j+1$.
    \end{enumerate}
The vertices $j,d,n-j$ and $n$ that appear in the first two conditions are distinct. The first two conditions imply that in any ASM-ordering, these four occur either in the order $d,n-j,j,n$ or $n,j,n-j,d$. Since the reverse of an ASM-ordering is an ASM-ordering, we may concentrate on the former case, and consider whether the remaining $n-4$ vertices can be inserted so that the conditions arising from Columns $j$ and $n-j$ also hold. 

The vertices $1,j+1,d+1$ and $n-j+1$ are distinct, and can be ordered so that conditions 3 and 4 are satisfied. If all eight vertices that appear in the above conditions are distinct, then the two sets of four can be ordered independently, and the arrangement can be completed to an ASM-ordering of the full vertex set.  

We note the only possible coincidences between the four-vertex sets $\{d,n-j,j,n\}$ and $\{1,j+1,d+1,n-j+1\}$, as follows.
$$
1=j,\  n-j+1=n,\ d+1=n-j,\ j+1=d.
$$
The first two possibilities above are equivalent, and so are the second two. All four cannot be satisfied simultaneously, since $j$ cannot be simultaneously equal to $1$ and $d-1$, as $d\ge 3$. The positions of $1$ and $n-j+1$, or of $d+1$ and $j+1$, may be constrained by the appearance of $d,n-j,j,n$ in that order in a candidate ASM-ordering. In the first case, $d+1$ and $j+1$ may be inserted freely and independently, in order to satisfy conditions 3 and 4. In the second case, the same applies to $1$ and $n-j+1$, hence the following statement.

\begin{theorem}\label{type 1 asm} 
Every $PT$-matrix of finite multiplicative order of Type 1 is permutation  similar to an alternating sign matrix.
\end{theorem}

\subsection{Type 2(c)}
By Theorem \ref{2c matrix}, a $PT$-matrix of Type 2(c) of finite order (or its transpose) is permutation similar to 
\begin{equation}\label{2c pq}
A=C_p\oplus C_q+T(1,n,h+1,n-l),
\end{equation}
where $1\le h< p,\ 1\le l<q$, and $\gcd (p,q)$ divides at least one of $h$ and $l$. We write $1,\dots ,n$ for the vertices of the graph $\Gamma$ determined by $A$, ordered according to the rows and columns of $A$. From Rows 1 and $h+1$ and Columns $n-l$ and $n$ of $A$, we observe that an ordering of the vertices $1,\dots ,n$ is an ASM-ordering if and only if it satisfies the following four conditions.
\begin{enumerate}
    \item $n-l$ occurs between $p$ and $n$;
    \item $n$ occurs between $n-l$ and $h$;
    \item $1$ occurs between $h+1$ and $n-l+1$;
    \item $h+1$ occurs between $1$ and $p+1$.
    \end{enumerate}
The four vertices that appear in the first two conditions are distinct, and so are the four that appear in conditions 3 and 4. If an ASM-ordering exists, then one exists in which the vertices $p,\ n-l, n, h$ occur in that order. We consider when conditions 3 and 4 are compatible with this constraint.    

Potential intersections of the four-vertex sets $\{n-l,p,n,h\}$ and $\{1,h+1,n-l+1,h+1\}$ are limited to the following possibilities:
$$
1=h,\ h+1=p,\ n-l+1=n,\ p+1=n-l. 
$$
If at most two of the above equalities hold, it is possible to insert the remaining vertices from $\{1,h,n-l+1,p+1\}$ to ensure that conditions 3 and 4 are satisfied. If any 3 of them hold however, conditions 3 and 4 are incompatible and cannot be simultaneously satisfied. This occurs in the following two cases.
\begin{enumerate}
    \item[Case 1] $h=1$ and $h+1=p$, and $l$ is either equal to $1$ or $q-1$. 
    In this case $p=2$, and the two vertices in the $q$-cycle that have indegree 3 are consecutive in the $q$-cycle.
    \item[Case 2] Alternatively, $l=1$ and $l=q-1$, and $h$ is either equal to $1$ or $p-1$. In this case $q=2$ and the two vertices of the $p$-cycle that have outdegree 3 are consecutive in the $p$-cycle. 
\end{enumerate}
We conclude as follows. 
\begin{theorem} \label{type 2c asm graph}
Let $\Gamma$ be a $PT$-graph of type 2(c) as defined in Theorem \ref{graphs}. Then there is an ordering of the vertices of $\Gamma$ whose corresponding $(0,1,-1)$-matrix is an ASM, unless one of the cycles in $\Gamma_P$ has length 2, and the two vertices of $\Gamma_T$ in the other cycle are consecutive in that cycle.
\end{theorem}

\begin{theorem}\label{type 2c asm matrix}
If the matrix $A$ of (\ref{2c pq}) has finite order, then it is permutation similar to an ASM, except in the following two cases
\begin{itemize}
    \item $p=2$, $q$ is odd, $q\ge 3$ and $l\in\{1,q-1\}$;
    \item $q=2$, $p$ is odd, $p\ge 3$ and $h\in\{1,p-1\}$.
\end{itemize}
\end{theorem}
The stipulation that $q$ or $p$ is odd in the two cases of Theorem \ref{type 2c asm matrix} arise from the finite order conditions in Theorem \ref{2c matrix}, and not from considerations of ASM-permutability.

\subsection{Type 2(d)}
By Theorem \ref{2d}, a $PT$-matrix of finite order of Type 2(d) (or its transpose) is permutation similar to 
\begin{equation}\label{type 2d matrix}
A=C_p\oplus C_q+ T(i_1,n,i_2,p),
\end{equation}
where $n=p+q$, $i_1\neq i_2$, $i_1\le p$, $i_2\le p$ and $i_1+i_2=p+2$, with some additional conditions that do not enter our analysis here. The condition $i_1+i_2=p+2$ ensures that $i_1\ge 2$ and $i_2\ge 2$. A vertex ordering of the corresponding graph is an ASM-ordering if and only if it satisfies the following conditions, arising respectively from Rows $i_1$ and $i_2$ and from Columns $p$ and $n$.
\begin{enumerate}
    \item $p$ occurs between $i_1-1$ and $n$;
    \item $n$ occurs between $i_2-1$ and $p$;
    \item $i_1$ occurs between $1$ and $i_2$;
    \item $i_2$ occurs between 1 and $p+1$.
\end{enumerate}
The four indices occurring in the first two conditions are distinct, and if an ASM-ordering exists, then one exists in which the entries $i_1-1,\ p,\ n,\ i_2-1$ occur in that order. Possible repetitions among the eight vertices that appear in the four conditions above are as follows.
\begin{equation}\label{four conditions}
    i_1=p\ \mathrm{or}\ i_1=i_2-1,\ 
    1=i_1-1\ \mathrm{or}\ 1=i_2-1,\ 
    i_2=p\ \mathrm{or}\ i_2=i_1-1,\ p+1=n.
\end{equation}
In each of the first three points above, the two possibilities are mutually exclusive. As in previous cases, if up to two elements of $\{i_1,1,i_2,p+1\}$ belong to $\{i_1-1,\ p,\ n,\ i_2-1\}$, an ASM-ordering of all $n$ vertices may be completed. 

First we suppose that $i_1$ and $i_2$ are not consecutive, so that $i_1\neq i_2-1$ and $i_2\neq i_1-1$. In this case $\{i_1,1,i_2,p+1\}$ and $\{i_1-1,\ p,\ n,\ i_2-1\}$ can intersect in at most three elements, and this occurs if and only if $\{i_1,i_2\}=\{2,p\}$ and $p+1=n$, so $q=1$. If $(i_1,i_2)=(2,p)$, then $1,i_2,p+1$ appear in that order and $i_1$ can be inserted between $1$ and $i_2$, so that all four requirements are satisfied. If $(i_1,i_2)=(p,2)$, then $i_2,p+1,1$ occur in that order, and $i_1$ cannot be inserted so that the third and fourth conditions are simultaneously satisfied. 

Suppose now that $i_1$ and $i_2$ are consecutive, and suppose first that $i_1=i_2-1$. Then if at least three of the four equalities of (\ref{four conditions}) are satisfied, either $1=i_1-1$ or $i_2=p$. Since $i_1+i_2=p+2$, each of these conditions implies that $p=3$. If either of them is satisfied then both are, and in this situation condition 3 is not satisfied; $i_1$ does not occur between $1$ and $i_2$. 

On the other hand if $i_2=i_1-1$ and three or more of the conditions of (\ref{four conditions}) hold, then $i_1=p=3$ and $i_2,i_1,1$ occur in that order, so that the third ordering condition is satisfied. The fourth can be satisfied by a suitable choice of position for $p+1$, provided that $p+1\neq n$ in which case the fourth condition cannot be satisfied. No ASM-ordering exists in the case $(i_1,i_2,q)=(3,2,1)$.

Our conclusion for $PT$-matrices of type 2(d) is as follows.
\begin{theorem}\label{type 2d asm matrix}
The $PT$-matrix $A$ of (\ref{type 2d matrix}) is permutation similar to an ASM, except where $q=1$, $i_2=2$ and $i_1=p$.
\end{theorem}

\subsection{Type 3(c)}
A $n\times n$ $PT$-matrix of Type 3(c) is permutation similar to 
\begin{equation}\label{type 3c matrix}
A = (C_p \oplus C_q\oplus C_m)+T(1,p+q,i,p+q+m),
\end{equation}
where $p,q$ and $m$ are positive integers with $p\ge 2$ and $p+q+m=n$, and $1<i\le p$. 
An ordering of the vertices $1,\dots ,n$ of the graph $\Gamma$ determined by $A$ is an ASM-ordering if and only if it satisfies the following conditions. 
\begin{enumerate}
    \item $n$ occurs between $p$ and $p+q$;
    \item $p+q$ occurs between $i-1$ and $n$;
    \item $i$ occurs between $1$ and $p+1$;
    \item $1$ occurs between $i$ and $p+q+1$.
\end{enumerate}
From the first two conditions we deduce that if an ASM-ordering exists, then one exists in which the vertices $p,\ n,\ p+q$ and $i-1$ occur in that order. The possible repetitions among the eight vertices that appear above are
$$
i=p,\ 1=i-1,\ p+1=p+q,\ p+q+1=n.
$$
If any three of the above equalities hold, then conditions 3. and 4. cannot be simultaneously satisfied by the insertion of the remaining element, hence the following conclusion on ASM-permutability for $PT$-matrices and graphs of type 3(c).

\begin{theorem}\label{type 3c asm matrix}
The matrix $A$ of (\ref{type 3c matrix}) is permutation similar to an ASM, except in the following three cases: 
\begin{itemize}
    \item $p=i=2$, and  $1\in\{q,m\}$;
    \item $i=p$ and $q=m=1$;
    \item $i=2$ and $q=m=1$.
\end{itemize} 
\end{theorem}

\section{Conclusion}
The results of this article identify all $PT$-matrices of finite order, whose associated graphs are weakly connected, or equivalently the property that every cycle of the permutation component includes a vertex that is incident with an arc corresponding to an entry of the $T$-block. Such examples can be augmented by the addition of permutation matrices as new diagonal blocks. We have observed that $n\times n$ elementary $PT$-matrices of types 1 and 2(c) may have finite multiplicative orders that do not occur in the symmetric group of degree $n$. 

It is not true that every alternating sign matrix of finite multiplicative order is permutation equivalent to a matrix direct sum of permutations and elementary $PT$-matrices, as the following example shows. 

$$
A=\left(\begin{array}{rrrrrr}
0 & 0 & 1 & 0 & 0 & 0 \\
0 & 1 & 0 & 0 & 0 & 0 \\
1 & -1 & 0 & 1 & 0 & 0 \\
0 & 1 & -1 & 0 & 0 & 1 \\
0 & 0 & 1 & -1 & 1 & 0 \\
0 & 0 & 0 & 1 & 0 & 0
\end{array}\right)
$$

This matrix $A$ has multiplicative order 12 and its minimum polynomial is 
$$\Phi_2(x)\Phi_{12}(x)=(x-1)(x^4-x^2+1).$$ It is clear that $A$ is not a $PT$-matrix since it has three negative entries; however it may be obtained from a permutation matrix by the addition of two $T$-blocks.

While there exist $6\times 6$ elementary $PT$-matrices of order 12, the analysis in Section \ref{algebra} confirms that none of their characteristic polynomials has $\Phi_{12}(x)$ as a factor. Thus $A$ is not similar to a $PT$-matrix. Since the symmetric group of degree 6 has no element of order 12, $A$ is not similar to a permutation matrix either. 

It would be of interest to know the maximum possible number of negative entries in a $n\times n$ ASM of finite multiplicative order. The maximum possible number of negative entries in an ASM of specified size occurs in the \emph{diamond ASMs}, which never have finite order, since their spectral radii exceed 1, as shown in \cite{spectra2}.

\bigskip
\textbf{Acknowledgement}\\ This research was partly supported by a grant from the Office of Equality, Diversity and Inclusivity at the National University of Ireland, Galway, whose support is gratefully appreciated.

\end{document}